\documentclass[11pt]{amsart}
\usepackage{amsfonts}
\usepackage{amsmath}
\usepackage{amssymb,amsthm,amscd}
\usepackage[mathscr]{eucal}
\usepackage{graphicx}
\usepackage{float}
\usepackage{color}
\usepackage[all]{xy}
\newtheorem{theorem}{Theorem}[section]
\newtheorem{lemma}{Lemma}[section]

\theoremstyle{remark}
\newtheorem{remark}{Remark}[section]
\theoremstyle{corollary}

\newtheorem{proposition}{Proposition}[section]

\title[Prescribed Gauduchon scalar curvature problem]{The prescribed Gauduchon scalar curvature problem in almost Hermitian geometry }

\author{Yuxuan Li}
\author{Wubin Zhou}
\author{Xianchao Zhou}

\address{School of Mathematical Science\\ Tongji University \\ Shanghai
200092, China} \email{1653454@tongji.edu.cn}

\address{  School of Mathematical Sciences\\ Tongji University \\
Shanghai 200092, China} \email{wbzhou@tongji.edu.cn}

\address{Department of Applied Mathematics\\ Zhejiang University of Technology\\ Hangzhou 310023,
China}\email{zhouxianch07@zjut.edu.cn}

\begin{document}
 \begin{abstract}
 In this paper we consider the prescribed Gauduchon scalar curvature problem on almost Hermitian manifolds. By deducing the expression of the Gauduchon scalar curvature under the conformal variation, the problem is reduced to solve a semi-linear partial differential equation with exponential nonlinearity. Using super and sub-solution method, we show that the existence of the solution to this semi-linear equation depends on the sign of a constant associated to Gauduchon degree. When the sign is negative, we give both necessary and sufficient conditions that a prescribed function is the Gauduchon scalar curvature of a conformal Hermitian metric. Besides, this paper recovers Chern  Yamabe problem, Lichnerowicz Yamabe problem and Bismut Yamabe problem. 
\end{abstract}
\maketitle

 
 \section{Introduction}

 Let $(M, J, h)$ be a compact almost Hermitian manifold of real dimension $2n$.  $J$ is an almost complex structure on the tangent bundle $TM$ which is compatible with the Hermitian metric $h$. Let $\nabla$ be the Levi-Civita connection with respect to $h$,   P. Gauduchon  in \cite{Gau2} introduced  one parameter family of canonical Hermitian connections $D^t$, namely, the Gauduchon connection as follows:
\begin{equation}\label{gco}
\begin{aligned}
&h(D_X^t Y,Z)
=h\bigl(\nabla_X Y-\frac{1}{2}J(\nabla_X J)Y,Z\bigr)\\
&+\frac{t}{4}h((\nabla_{JY}J) Z+J(\nabla_Y J)Z,X)-\frac{t}{4}h((\nabla_{JZ}J) Y+J(\nabla_Z J)Y,X),
\end{aligned}
\end{equation}
where $X, Y, Z$ are smooth vector fields on $M$ and $t\in \mathbb{R}$. Here are three important cases: $D^0$ is the first canonical Hermitian connection, also called the Lichnerowicz connection \cite{Kob};
$D^1$ is the second canonical Hermitian connection, also called the Chern connection as  it coincides
with the connection used by S. S. Chern \cite{Chern} in the integrable case;
$D^{-1}$ is the Bismut connection which, in the integrable case, is characterized by its torsion being skew-symmetric \cite{Bis}.

 Let $K^t$ be the curvature tensor and then one can define the \textbf{Gauduchon scalar curvature} by
$$s(t)=\sum_{i,j} K^t(u_{\bar{i}}, u_i, u_j, u_{\bar{j}})$$
where $\{u_i,u_{\bar{i}}\}_{i=1,\cdots,n}$ is a local unitary frame field adapted to almost Hermitian complex structure $J$. Since the conformal variation of the Hermitian metric $h$ does keep the compatibility with the almost complex structure $J$, it is natural to propose Gauduchon Yamabe problem: in a given conformal class $\{e^{u}h, u\in C^{\infty}(M)\}$, one would like to find a Hermitian metric with constant Gaucuchon scalar curvature. 
 When $J$ is integrable, $s(1)$ is Chern scalar curvature and this problem is  called Chern Yamabe problem \cite{ACS} . Whereas $s(-1)$ is Bismut scalar curvature and it is called Bismut Yamabe problem \cite{Bar}.   E. Fusi in \cite{Fus} extended Chern Yamabe problem to more general setting the prescribed Chern scalar curvature problem (also see \cite{Ho} for balanced background metrics).  For other similar prescribed scalar curvature problems on almost Hermitian manifolds we recommend the recent work \cite{GZ} by Zhou and Ge. 
 
In this paper we will study the following \textbf{prescribed Gauduchon scalar curvature problem}:
 \begin{quotation}
 For a given smooth function $\hat{s}(t)$ on the almost Hermitian manifold $(M, J, h)$, does $M$ admit  a  conformal Hermitian metric  $e^{u}h$ with Gauduchon scalar curvature $ \hat{s}(t)$?
 \end{quotation}
 
 In order to handle this problem, we first use the moving frame method to deduce the expression of  Gauduchon scalar  curvature  $\hat{s}(t)$ under the conformal variation $e^u h$ in Section 2. We  have (see Proposition \ref{cgsc1})
 \begin{equation}\label{egs}
 e^{u}\hat{s}(t) = s(t)+\frac{nt-t+1}{2} \Delta^{Ch} u.
 \end{equation}
Here $\Delta^{Ch}$ is Chern Laplace operator defined locally as $\Delta^{Ch}:=-2 h^{i\bar{j}}\partial_i\bar{\partial}_j$ and 
 $$\Delta^{Ch} =\Delta_d +\langle\alpha_F, d\cdot \rangle$$
 with Hodge Laplace operator $\Delta_d$ and the Lee form $\alpha_F$ .  
  
The variation of the Gauduchon scalar curvature is the same as the case when $J$ is integrable,  which has been obtained by G. Barbaro \cite{Bar} by using holomorphic coordinate system. For special case $t=1$, it is also gotten by M. Lejmi and M. Upmeier \cite{LeU} for almost integrable case  and by D. Angella, S. Calamai and C. Spotti \cite{ACS} for integrable case (for K\"ahler case by M.S. Berger in \cite{Berg}).     Then the prescribed Gauduchon scalar problem is transferred to solve the following semi-linear equation
\begin{equation}\label{seu}
\Delta_d u +\left\langle\alpha_F,du \right\rangle+\frac{2}{nt-t+1} s(t)=\frac{2\hat{s}(t)}{nt-t+1}  e^{u}
\end{equation}
 when $nt-t+1\neq 0$.
  Obviously when $t=\frac{1}{1-n}$, the equation (\ref{egs}) has a solution if and only if  $s(t)/\hat{s}(t)$ is positive and well defined at the zero of $\hat{s} (t)$. So we only consider the solvability (\ref{seu}) with $nt-t+1\neq 0$.
 
  Recall that in each conformal class $\{e^{u}h, u\in C^{\infty}(M)\}$ there is a unique Gauduchon metric  with the unit volume (see \cite{Gau1}).  For simplicity we can assume the background metric $h$ is just the Gauduchon metric with unit volume in this paper. The advantage is that,  for any smooth function $u$, 
  $$\int_M \left\langle\alpha_F, du\right\rangle dV  =0$$
where $dV$ is the volume form. This advantage ensures the solvability of the linear equation \eqref{leg} below.
  
We prefer to transfer (\ref{seu}) to another more concise form.  Let $g$ be a solution of the following linear equation
  \begin{equation}\label{leg}
  \Delta_d g + \left\langle \alpha_F,dg \right\rangle+\frac{2}{nt-t+1}  s(t)=\frac{2}{nt-t+1} \int_M s(t) dV. 
  \end{equation}
 Set $w=u-g$,  combine (\ref{seu}) and (\ref{leg}), the equation (\ref{seu}) can be transferred to the following crucial equation 
\begin{equation}\label{edw}
\Delta_d w+\left\langle\alpha_F, dw\right\rangle+c(t)=\varphi e^w
\end{equation}
with $$c(t)=\frac{2}{nt-t+1} \int_M s(t)  dV\ \  \text{and}\ \  \varphi=\frac{2e^g }{nt-t+1}\hat{s}(t).$$
It turns out that  the solvability of (\ref{edw}) depends on the sign of the constant $c(t)$ for each fixed $t$.  In fact, 
 \begin{equation}
  c(t)=\frac{2}{nt-t+1} \Gamma(t)
 \end{equation}
 and $\Gamma(t)= \int_M s(t)  dV$ can be regarded as a conformal invariant which depends on the almost complex structure and the conformal class.  

If the Gauduchon metric $h$ is balanced and then the Lee form $\alpha_F$ vanishes, the equation \eqref{edw} is just Kazdan -Warner equation $\Delta_d w +c=\varphi e^w$ (see \cite{KW, KW1}).  Because of this, we call \eqref{edw} as Kazdan-Warner type equation.  Kazdan-Warner type equation has been concerned by E. Fusi in \cite{Fus} to study the prescribed Chern scalar curvature problem on Hermitian manifolds  for $t=1$. Since the objective of this paper is to generalize all the result mentioned before, some of conclusions will recover E. Fusi's work. To avoid such repetition, we will study this equation more in depth form the point of view of both geometry and differential equations. For example, one can see the statements (a) and (d) of Theorem \ref{cphs}.

If $c(t)=0$ and $\hat{s}(t)=0$, then the equation (\ref{edw}) is linear.  We thus can solve the  Gauduchon Yamabe problem for zero Gauduchon scalar curvature as follows.
\begin{theorem}
If Gauduchon degree is zero, there are Hermitian metrics conformal to $h$ with zero Gauduchon scalar curvature.
\end{theorem}

When $c(t)<0$, we can prove the following theorem
 \begin{theorem} [See Remark \ref{re1}]\label{thm11}
Let $(M, J, h)$ be a compact almost Hermitian manifold of real dimenson $2n\geq 4$ with negative $c(t)$. Then for each given smooth function $\hat{s}(t)$, there is an associated constant $c(\hat{s}(t))$. If $c(\hat{s}(t))<c(t)<0$, then there is a Hermitian metric conformal to $h$ with Gauduchon scalar curvature  $\hat{s}(t)$. Whereas $c(t)<c(\hat{s}(t))$, it is impossible that $\hat{s}(t)$ is the Gauduchon scalar curvature function of a Hermitian metric in the conformal class of $h$. 
 \end{theorem}

Through the deep understanding of $c(\hat{s}(t))$,  we have the following sufficient  condition  for the  prescribed Gauduchon scalar curvature problem.
 \begin{theorem}[See Remark \ref{re2}]\label{thm12}
When $c(t)<0$,  $c(\hat{s}(t))=-\infty$ if and only if the nonzero function $\frac{\hat{s}(t)}{nt-t+1}  $ is nonpositive. That's to say, under the assumption $c(t)<0$, any nonzero smooth function $\hat{s}(t)$ with $\frac{\hat{s}(t)}{nt-t+1} \le 0$ is the Gauduchon scalar curvature of a unique Hermitian metric conformal to $h$. Besides,  the necessary condition $\int_M \frac{2e^g }{nt-t+1}\hat{s}(t)dV<0$ is not sufficient.
 \end{theorem}

The cases $c(t)>0$ or $c(t)=0$ with $\hat s(t)\neq 0$ are difficult since it is hard to obtain $C^0$  estimate. When $M$ is a Riemannian surface, one can use Moser-Trudinger inequality and the variational method as \cite{KW}. But when $n\geq 2$, the equations (\ref{seu}) or (\ref{edw}) are both super critical equations.  In fact, when $c(t)>0$, one can construct a sequence of functions to prove that the energy functional is not bounded below when using variational method (In \cite{CSZ}, a family of Lipschitz functions were constructed to prove that in this case the lower bound does not exist).  It reminds us to study these cases through other approaches such as fixed point theorem (Theorem \ref{smallosc}) and implicit theorem(Theorem \ref{implicit}).

The paper is organized as mentioned above. In Section 2, we recall almost Hermitian geometry and define two type scalar curvatures $s(t)$ and $s_2(t)$. By using moving frame method, we deduce the variation of curvature tensors under the conformal change and then get the conformal  variations of $s(t)$ and $s_2(t)$.  In Section 3, we give some preliminaries related to Kazdan-Warner type equation which will be used in Section 4.  Section 4 is our main part to study Kazdan-Warner type equation with  $c(t)<0$, we will prove Theorem \ref{thm11} and Theorem \ref{thm12}.  In the last section 5,  we will use implicit theorem and fixed point theorem to solve the Kazdan-Warner type equation under different assumptions but regardless of the sign of $c(t)$. 

 \section{Gauduchon scalar curvatures under conformal variations}

Let $(M,  J,  h)$ be a compact almost Hermitian manifold of complex dimension $n$ with almost complex structure $J$ and Hermitian metric $h$.  Here $h$ is compatible with the almost complex structure $J$. Then  the fundamental 2-form $F$ associated to $h$ and $J$ is defined by
\begin{equation*}
F(X,Y)=h(JX,Y)
\end{equation*}
for all $X,Y\in\Gamma(TM)$ and $TM$ is the tangent bundle of $M$. 
    Let $\delta^h$ be the codifferential operator, another important differential form on the almost Hermitian manifold $M$ is the Lee form defined by $$\alpha_F=J\delta^h F.$$ 
The metric $h$ is called Gauduchon  if  $\delta^h\alpha_F=0$.

With the Gauduchon connection (\ref{gco}) and the curvature tensor $K^t$, we can define two scalar curvatures $s(t)$ and $s_2(t)$ (see the equation (\ref{defst}) below).   We will exploit the moving frame method to deduce the variation of the scalar curvature, which turns out to be very effective in our study of curvatures in almost Hermitian geometry.

Following the formations in \cite{FuZ,Kob}, we first recall the structure equations of canonical connections $D^0$ and $D^1$. For convenience, we fix the index range $1\leq i,j,k,\dots\leq n$ and $1\leq A,B,C,\dots\leq 2n$. We use the Einstein summation convention, i.e., the repeated indices are summed over.

Let $\{e_i, e_{n+i}=Je_i\}_{i=1,2,...,n}$ be a local $J$-adapted orthonormal frame field on $(M,J,h)$. Its dual coframe field is denoted by
$\{\omega^1,\omega^2,...,\omega^{2n}\}$. Hence a unitary frame field is
$\{u_i=\frac{1}{\sqrt{2}}(e_i-\sqrt{-1}e_{n+i})\}_{i=1,2,\dots,n}$. Set $u_{\bar{i}}=\overline{u_i}$. We denote the unitary coframe field by $\{\theta^i=\frac{1}{\sqrt{2}}(\omega^i+\sqrt{-1}\omega^{n+i})\}_{i=1,2,\dots,n}$.
Let $\omega=(\omega_B^A)$ and $\Omega=(\Omega_B^A)$ be, respectively,
the connection and curvature form matrices of the Levi-Civita connection. Its structure equations are
\begin{equation*}
\begin{cases}
&d\omega^A=-\omega_B^A\wedge\omega^B,\\
&d\omega_B^A=-\omega_C^A\wedge\omega_B^C+\Omega_B^A,
\end{cases}
\end{equation*}
where $\omega_B^A+\omega_A^B=0$, $\Omega_B^A+\Omega_A^B=0$.

Set 
$\varphi_j^i=\frac{1}{2}(\omega_j^i+\omega_{n+j}^{n+i})+\frac{\sqrt{-1}}{2}(\omega_j^{n+i}-\omega_{n+j}^i)$ and $\mu_j^i=\frac{1}{2}(\omega_j^i-\omega_{n+j}^{n+i})+\frac{\sqrt{-1}}{2}(\omega_j^{n+i}+\omega_{n+j}^i).$
Then locally, $\varphi=(\varphi_j^i)$ is the matrix of the Lichnerowicz connection form. The structure equations of the Lichnerowicz connection are
\begin{equation}\label{csc2}
\begin{cases}
&d\theta^i=-\varphi_j^i\wedge\theta^j+\tau^i, \\
&d\varphi_j^i=-\varphi_k^i\wedge\varphi_j^k+\Phi_j^i,
\end{cases}
\end{equation}
where the $\tau^i=-\mu_j^i\wedge \overline{\theta^j}$ are the torsion forms and the $\Phi_j^i$ are the curvature forms.

The Chern connection is the unique Hermitian connection (making $J$ and $h$ parallel) whose torsion form has vanishing $(1,1)$-component. Locally, let $\psi=(\psi_j^i)$ be the matrix of the Chern connection form. Then the structure equations of the Chern connection are
\begin{equation}\label{csc4}
\begin{cases}
&d\theta^i=-\psi_j^i\wedge\theta^j+T^i,\\
&d\psi_j^i=-\psi_k^i\wedge\psi_j^k+\Psi_j^i,
\end{cases}
\end{equation}
where the $T^i$ are the torsion forms and the $\Psi_j^i$ are the curvature forms. As $T^i$ has no $(1,1)$-component, it can be written as
\begin{equation}\label{csc5}
T^i=\frac{1}{2}T_{jk}^i \theta^j\wedge\theta^k+\frac{1}{2}T_{\bar{j}\bar{k}}^i \overline{\theta^j}\wedge\overline{\theta^k}
\end{equation}
where $T_{jk}^i+T_{kj}^i=0$ and $T_{\bar{j}\bar{k}}^i+T_{\bar{k}\bar{j}}^i=0$. 
Using the Chern  torsion,  one can see the Lee form $\alpha_F$ can be rewritten as
\begin{equation}\label{csc6}
\alpha_F=T_{ji}^i\theta^j+\overline{T_{ji}^i}\overline{\theta^j},
\end{equation}
which will be useful later. Combing the torsion terms in (\ref{csc2}) and (\ref{csc4}), we have
\begin{equation}\label{csc7}
(\varphi_j^i-\psi_j^i)\wedge \theta^j+T^i-\tau^i=0.
\end{equation}
Since $\varphi_j^i+\overline{\varphi_i^j}=0$, $\psi_j^i+\overline{\psi_i^j}=0$, and $\mu_j^i+\mu_i^j=0$,
by comparing the types of forms in (\ref{csc7}), we get
\begin{equation}\label{csc8}
\varphi_j^i-\psi_j^i=\frac{1}{2}T_{jk}^i \theta^k-\frac{1}{2}\overline{T_{ik}^j} \overline{\theta^k}.
\end{equation}

Set $\gamma=(\gamma_j^i)=(\varphi_j^i-\psi_j^i)$. Locally, let $\psi(t)=(\psi_j^i(t))$ be the connection form matrix of the Gauduchon connections $D^t$. Then
$\psi(t)=\varphi-t\gamma$. The structure equations of $D^t$ are
\begin{equation}
\begin{cases}
&d\theta^i=-\psi_j^i(t)\wedge\theta^j+T^i(t),\\
&d\psi_j^i(t)=-\psi_k^i(t)\wedge\psi_j^k(t)+\Psi_j^i(t),
\end{cases}
\end{equation}
where the $T^i(t)=tT^i+(1-t)\tau^i$ are the torsion forms and the $\Psi_j^i(t)$ are the curvature forms. More precisely, the curvature forms $\Psi_j^i(t)$ can be written as
\begin{equation}
\Psi_j^i(t)=\frac{1}{2}K_{\bar{i}jkl}^t\theta^k\wedge\theta^l+K_{\bar{i}jk\bar{l}}^t\theta^k\wedge\overline{\theta^l}
+\frac{1}{2}K_{\bar{i}j\bar{k}\bar{l}}^t\overline{\theta^k}\wedge\overline{\theta^l}
\end{equation}
where $K_{\bar{i}jk\bar{l}}^t=K^t(u_{\bar{i}},u_j,u_k,u_{\bar{l}})$ and the others are similar.

By contracting the curvature tensor
$K^t$ of the  connection $D^t$ with $h$, two Hermitian scalar curvatures $s(t)$ and $s_2(t)$ are defined. In fact, for a given unitary frame field $\{u_i\}_{i=1,2,\dots,n}$, one can see 
\begin{equation} \label{defst}
s(t)= K^t(u_{\bar i}, u_i,u_j,u_{\bar j})=K^t_{\bar{i} i j \bar{j}}\quad  \textup{and}\quad  s_2(t)=K^t(u_{\bar i},u_j,u_i,u_{\bar j})=K^t_{\bar{i} ji\bar{j}}
\end{equation}
where the repeated indices are summed over.

Let $\hat{h}=e^{2f}h$ be a conformal change of the metric $h$,  where $f\in C^\infty(M)$.
It is obvious that $\{\hat{\theta}^i=e^f \theta^i\}_{i=1,2,\dots,n}$ give the corresponding local unitary coframe field on $(M,J,\hat{h})$.
We use the superscript $\hat{}$ to denote quantities related to the metric $\hat{h}$.

Next, we consider  the variations of curvature under a conformal change of
the Hermitian metric $h$. For the Lichnerowicz connection $D^0$
and the Chern connection $D^1$, S. Kobayashi \cite{Kob} obtained the explicit transformations
of the corresponding connection forms under a conformal change of the metric $h$, respectively.

Define $f_i$ and $f_{\bar{i}}$ by
$$df=\partial f+\bar{\partial} f=f_i\theta^i+f_{\bar{i}}\overline{\theta^i},$$
where $\partial f=f_i\theta^i$ is the $(1,0)$-part of $df$, and $\bar{\partial} f=f_{\bar{i}}\overline{\theta^i}$ is the $(0,1)$-part of $df$.
From the Theorem 5.2 in \cite{Kob}, the structure equations of the Lichnerowicz connection $\hat{D}^0$ of $\hat{h}$ are
\begin{equation}\label{clse}
\begin{cases}
&d\hat{\theta}^i=-\hat{\varphi}_j^i\wedge\hat{\theta}^j+\hat{\tau}^i,\\
&d\hat{\varphi}_j^i=-\hat{\varphi}_k^i\wedge\hat{\varphi}_j^k+\hat{\Phi}_j^i,
\end{cases}
\end{equation}
where $\hat{\varphi}_j^i=\varphi_j^i+f_j\theta^i-f_{\bar{i}}\overline{\theta^j}$
and $\hat{\tau}^i=e^f(\tau^i+f_{\bar{i}}\theta^j\wedge\overline{\theta^j}+\bar{\partial} f\wedge \theta^i).$

From Theorem 4.1 in \cite{Kob}, the structure equations of the Chern connection $\hat{D}^1$ of $\hat{h}$ are
\begin{equation}\label{ccse}
\begin{cases}
&d\hat{\theta}^i=-\hat{\psi}_j^i\wedge\hat{\theta}^j+\hat{T}^i,\\
&d\hat{\psi}_j^i=-\hat{\psi}_k^i\wedge\hat{\psi}_j^k+\hat{\Psi}_j^i,
\end{cases}
\end{equation}
where $\hat{\psi}_j^i=\psi_j^i+(\partial f-\bar{\partial} f)\delta_j^i$ and $\hat{T}^i=e^f(T^i+2\partial f\wedge \theta^i).$

Using above structure equations (\ref{clse}) and (\ref{ccse}), we get the following structure equations of $\hat{D}^t$,
\begin{equation}\label{cgsc}
\begin{cases}
&d\hat{\theta}^i=-\hat{\psi}_j^i(t)\wedge\hat{\theta}^j+\hat{T}^i(t),\\
&d\hat{\psi}_j^i(t)=-\hat{\psi}_k^i(t)\wedge\hat{\psi}_j^k(t)+\hat{\Psi}_j^i(t),
\end{cases}
\end{equation}
with $\hat{\psi}_j^i(t)=(1-t)\hat{\varphi}_j^i+t\hat{\psi}_j^i$ and $\hat{T}^i(t)=(1-t)\hat{\tau}^i+t \hat{T}^i$.  Moreover,
\begin{equation*}
\hat{\psi}_j^i(t)=\psi_j^i(t)+(1-t)(f_j\theta^i-f_{\bar{i}}\overline{\theta^j})+t(\partial f-\bar{\partial} f)\delta_j^i,
\end{equation*}
then together with (\ref{cgsc}), by direct computations, we have
\begin{align}\label{ce18}
\hat{\Psi}_j^i(t)&=\Psi_j^i(t)+td(\partial f-\bar{\partial}f)\delta_j^i
+(1-t)^2(f_k\theta^i-f_{\bar{i}}\overline{\theta^k})\wedge(f_j\theta^k-f_{\bar{k}}\overline{\theta^j}) \notag\\
&~~~~~~~~~~~~+(1-t)^2((\varphi_k^i-\psi_k^i)\wedge(f_j\theta^k-f_{\bar{k}}\overline{\theta^j})
+(f_k\theta^i-f_{\bar{i}}\overline{\theta^k})\wedge(\varphi_j^k-\psi_j^k))\notag\\
&~~~~~~~~~~~~+(1-t)(d(f_j\theta^i-f_{\bar{i}}\overline{\theta^j})+\psi_k^i\wedge(f_j\theta^k-f_{\bar{k}}\overline{\theta^j})
+(f_k\theta^i-f_{\bar{i}}\overline{\theta^k})\wedge \psi_j^k).
\end{align}

Using  the Chern connection, define $f_{jk}$, $f_{j\bar{k}}$, $f_{\bar{j}k}$ and $f_{\bar{j}\bar{k}}$ by
\begin{equation}\label{df19}
df_j-f_k\psi_j^k=f_{jk}\theta^k+f_{j\bar{k}}\overline{\theta^k}, ~~df_{\bar{j}}-f_{\bar{k}}\overline{\psi_j^k}=f_{\bar{j}k}\theta^k+f_{\bar{j}\bar{k}}\overline{\theta^k}.
\end{equation}
Note that $f_{j\bar{k}}=f_{\bar{k}j}$ and $2f_{i\bar{i}}=\langle dJdf, F\rangle$ for the Chern connection \cite{TWY}. In fact, $$-2f_{i\bar{i}}=\Delta^{Ch}f$$
where $\Delta^{Ch}$ is the Chern Laplacian with respect to the metric $h$ as in \cite{ACS,LeU}. Moreover,
$\Delta^{Ch}$ is related to the Hodge Laplacian $\Delta_d=d\delta^h+\delta^hd$ by
\begin{equation}\label{chlap}
-2f_{i\bar{i}}=\Delta^{Ch}f=\Delta_d f+\langle \alpha_F,df\rangle.
\end{equation}

By using the structure equation (\ref{csc4}) of the Chern connection, together with (\ref{df19}), we have
\begin{align}\label{E1}
d(\partial f-\bar{\partial}f)&=(df_i-f_j\psi_i^j)\wedge\theta^i-(df_{\bar{i}}-f_{\bar{j}}\overline{\psi_i^j})\wedge\overline{\theta^i}
+f_iT^i-f_{\bar{i}}\overline{T^i}\notag\\
&=(f_{ij}\theta^j+f_{i\bar{j}}\overline{\theta^j})\wedge\theta^i-(f_{\bar{i}j}\theta^j+f_{\bar{i}\bar{j}}\overline{\theta^j})\wedge\overline{\theta^i}
+f_iT^i-f_{\bar{i}}\overline{T^i},
\end{align}
and
\begin{align}\label{E2}
&d(f_j\theta^i-f_{\bar{i}}\overline{\theta^j})+\psi_k^i\wedge(f_j\theta^k-f_{\bar{k}}\overline{\theta^j})
+(f_k\theta^i-f_{\bar{i}}\overline{\theta^k})\wedge \psi_j^k \notag\\
&~~~~~~~=(df_j-f_k\psi_j^k)\wedge\theta^i-(df_{\bar{i}}-f_{\bar{k}}\overline{\psi_i^k})\wedge\overline{\theta^j}
+f_jT^i-f_{\bar{i}}\overline{T^j} \notag \\
&~~~~~~~=(f_{jk}\theta^k+f_{j\bar{k}}\overline{\theta^k})\wedge\theta^i
-(f_{\bar{i}k}\theta^k+f_{\bar{i}\bar{k}}\overline{\theta^k})\wedge\overline{\theta^j}
+f_jT^i-f_{\bar{i}}\overline{T^j}.
\end{align}

Put (\ref{E1}), (\ref{E2}), (\ref{csc5}) and (\ref{csc8}) into the expression of the curvature terms $\hat{\Psi}_j^i(t)$ in (\ref{ce18}), then comparing the types of forms in both sides of the obtained formula,  we can get the relation  between  the curvature components $\hat{K}_{\bar{i}jk\bar{l}}^t$ (resp. $\hat{K}_{\bar{i}jkl}^t$, $\hat{K}_{\bar{i}j\bar{k}\bar{l}}^t$)
and $K_{\bar{i}jk\bar{l}}^t$ (resp. $K_{\bar{i}jkl}^t$, $K_{\bar{i}j\bar{k}\bar{l}}^t$). Here, we only give the explicit formula of $\hat{K}_{\bar{i}jk\bar{l}}^t$ as follows,
\begin{align}
e^{2f}\hat{K}_{\bar{i}jk\bar{l}}^t&=K_{\bar{i}jk\bar{l}}^t-t(f_{k\bar{l}}+f_{\bar{l}k})\delta_j^i
+(1-t)^2(f_{\bar{i}}f_j\delta_l^k-f_pf_{\bar{p}}\delta_k^i\delta_l^j)\\
&\quad  \quad +\frac{(1-t)^2}{2}(f_{\bar{i}}T_{jk}^l+f_j\overline{T_{il}^k}
-f_p\overline{T_{pl}^j}\delta_k^i-f_{\bar{p}}T_{pk}^i\delta_l^j)\notag\\
&\quad \quad -(1-t)(f_{j\bar{l}}\delta_k^i+f_{\bar{i}k}\delta_l^j).\notag
\end{align}

Thus, using contractions, and combing (\ref{csc6}) and (\ref{chlap}), the two Hermitian scalar curvatures $\hat{s}(t)$ and $\hat{s}_2(t)$ satisfy
\begin{equation*}
e^{2f}\hat{s}(t)=K_{\bar{i}ij\bar{j}}^t-(nt+1-t)(2f_{i\bar{i}})
=s(t)+(nt-t+1)(\Delta_d f+\langle\alpha_F,df\rangle),
\end{equation*}
\begin{align*}
e^{2f}\hat{s}_2(t)&=K_{\bar{i}ji\bar{j}}^t+(n(1-t)+t)(-2f_{i\bar{i}})
+(1-t)^2(1-n^2)f_if_{\bar{i}}\\
& \ \ \ \ -(n+1)\frac{(1-t)^2}{2}(f_i\overline{T_{ij}^j}+f_{\bar{i}}T_{ij}^j)\notag\\
&=s_2(t)+(n(1-t)+t)\Delta_d f+\frac{(1-t)^2(1-n^2)}{2}\langle df, df\rangle\notag\\
&\ \ \ \ +(n(1-t)+t-(n+1)\frac{(1-t)^2}{2})\langle \alpha_F, df\rangle.\notag
\end{align*}

Finally we have the following proposition.
\begin{proposition}\label{cgsc1}
Let $(M,J,h)$ be an almost Hermitian manifold of complex dimension $n$. $\hat{h}=e^{2f}h$ is a conformal change of the metric $h$,  where $f\in C^\infty(M)$. $s(t)$ and $s_2(t)$  (resp. $\hat{s}(t)$ and $\hat{s}_2(t)$) are the two Hermitian scalar curvatures
of the Gauduchon connections $D^t$  of $h$ (resp. $\hat{D}^t$ of $\hat{h}$). Then
\begin{equation}\label{pros1}
e^{2f}\hat{s}(t)=s(t)+(nt-t+1)(\Delta_d f+\langle\alpha_F,df\rangle),
\end{equation}
\begin{align}
e^{2f}\hat{s}_2(t)&=s_2(t)+(n(1-t)+t)\Delta_d f+\frac{(1-t)^2(1-n^2)}{2}\langle df, df\rangle\\
&~~~~~~~+(n(1-t)+t-(n+1)\frac{(1-t)^2}{2})\langle \alpha_F, df\rangle,\notag
\end{align}
where $\alpha_F$ is the Lee form and $\Delta_d$ is the Hodge Laplacian of $h$, respectively.
\end{proposition}

Let $u=2f$, we have scalar curvature equation (\ref{leg}). Then the prescribed Gauduchon curvatures problem can be transformed into solving the equation (\ref{leg}).  

In the end,  we introduce the notation of Gauduchon degree. In \cite{Gau} P. Gauduchon proved in each conformal class $[h]=\{e^{2f} h| f \in C^{\infty}(M)\}$ there is a unique Gauduchon metric up to a constant multiple.  Still we denote $h$ by the unique Gauduchon metric $h_0$ in the conformal class $[h]$ with unit volume form, then motivated by the Gauduchon degree in Hermitian geometry \cite{Gau1,ACS}, we
introduce the following two conformal invariants:
\begin{equation*}
\Gamma(t)=\Gamma(M,J,[h],D^t)=\int_M s(t)\frac{F^n}{n!}
\end{equation*}
and
\begin{equation*}
\Gamma_2(t)=\Gamma(M,J,[h],D^t)=\int_M s_2(t)\frac{F^n}{n!}.
\end{equation*}

\section{Some preliminaries related to Kazdan-Warner type equation}

In this section we give some preliminaries related to Kazdan and Warner's work in \cite{KW,KW1}.

We show the solvability of the linear equation (\ref{leg}) in Lemma \ref{sec2} below, then for a fixed $t$ we can transfer the prescribed Gauduchon scalar curvature problem to solve the following semi-linear differential equation, namely Kazdan-Warner type equation (\ref{edw}) on $M$
 \begin{equation}\label{KWte}
\Delta_{d} w +\left\langle\alpha_F, dw\right\rangle +c = \varphi e^{w}.
\end{equation}

When the dimension of $M$ is $2$,  $\alpha_F$ will vanish, and then the equation (\ref{edw}) or (\ref{KWte}) becomes Kazdan-Warner equation 
which has been deeply studied by Kazdan and  Warner in \cite{KW, KW1}  to discuss Gaussian curvature problem on compact Riemannian surface and $c$ is the Euler characteristic of the compact Riemannian surface. Also they had considered the solvability as $n\geq 2$ from the perspective of differential equations when $c<0$. 

If we directly use the results of Kazdan and Warner we can only get the results for the balanced case, in which the gradient term $\left\langle\alpha_F, d w\right\rangle$ vanishes, but if we add a gradient term, we don't have to restrict our perspective on the balanced case.  In this section we give some preliminaries.

The solvability of the linear equation \eqref{leg}  follows the theorem below.
\begin{lemma}\label{sec2}
Let $(M,J,h)$ be a compact almost Hermitian manifold of real dimension $2 n$. If $f \in C^{\infty}(M)$ satisfies $\int_{M}f dV = 0$, 
 $$\Delta_d g+\langle \alpha_F, d g\rangle=f $$
 has a unique smooth solution up to adding a constant.
\end{lemma}
\begin{proof}
The proof is the same as Theorem 3.1 in \cite{ACS}. Notice that the adjoint of $\Delta _d+ \langle \alpha_F, d\cdot\rangle$  is $\Delta _d- \langle \alpha_F, d\cdot\rangle$.   By integration, it is easy to show  the kernel of the adjoint operator equals constants. Then by standard linear PDF theory (one can see  \cite[Theorem 1.5.3]{Joyce}),  the co-kernel of $\Delta_d g+\langle \alpha_F, d g\rangle: C^{2,\alpha}(M)\to C^{0,\alpha}(M)$ consists of constants.  It implies $\Delta_d g+\langle \alpha_F, d g\rangle=f $ is solvable and the solution is smooth when $\int_{M}f dV = 0$.  Also the kernel of $\Delta _d+ \langle \alpha_F, d\cdot\rangle$ equals constants which indicates the uniqueness of the solution up to adding a constant. 
\end{proof}

Then we give some lemmas for regularity of the special case for elliptic regularity when the kernel of linear second order operator is $0$ and its simple corollary using Sobolev inequalities. Denote $\|\cdot\|_p$ the norm in Sobolve space $L^p(M)$, $\|\cdot\|_\infty$ the uniform norm on $M$ and $\|\cdot\|_{k,p}$ the norm in the Sobolev space $W^{k,p}(M)$ of the function whose derivatives up to order $k$ are all in $L^p(M)$. 

\begin{lemma}\label{lemma1}
For operator $L u \equiv \Delta_{d} u+\left\langle\alpha_{F}, d u\right\rangle-c u$ with $c<0$, there exists $C$ and $\gamma$ such that for any $u \in W^{2,p}(M)$
\begin{equation}\label{1}
\|u\|_{2, p} \leq C\|L u\|_{p}
\end{equation}
and
\begin{equation}\label{2}
\|u\|_{\infty}+\|\nabla u\|_{\infty} \leq \gamma\|L u\|_{p},
\end{equation}
here $p > dim(M)$.
\end{lemma}
\begin{proof}
In the proof $C$ denotes a constant that may change from line to line, but is always independent of $u$. Since $L$ is a linear second order elliptic operator with smooth coefficient, we have the following $L^{p}$-regularity that for any $p>1$ there is a constant $C>0$ such that for all $u \in W^{2,p}(M)$
\begin{equation}\label{3}
\|u\|_{2, p} \leq C\left(\|L u\|_{p}+\|u\|_{ p}\right).
\end{equation}
Using maximum principle one can easily obtain that $\operatorname{ker}L=0$, we then need an inequality that if $\operatorname{ker}L=0$, then for all $u \in W^{2,p}(M)$ 
\begin{equation}\label{4}
\|u\|_{p} \leq C\|L u\|_{p}.
\end{equation}
This can be simply proved by contradiction. Assume that for any $n>0$ there exists $u_{n}$ such that
$$
\left\|u_{n}\right\|_{p}=1 \text { and }\left\|L u_{n}\right\|_{p} \leq 1 / n ,
$$
thus, $\left\|L u_{n}\right\|_{p} \rightarrow 0$, in particular, $\left\|u_{n}\right\|_{2,p}$ is bounded since $\left\|u_{n}\right\|_{p}+\left\|L u_{n}\right\|_{p}$ is bounded.   From the compactness of the imbedding of $W^{2,p}(M)$ to  $L^{p}(M)$ and the weak compactness of bounded sets  in $W^{2,p}(M)$, there is a subsequence  relabel  as $v_{n}$ which is weak convergent in $W^{2,p}(M)$ to some $u$ satisfying $\|u\|_p=1$. Obviously $\int_M g L v_n dV\to \int_M g L u dV$ for all $g\in L^{\frac{p}{p-1}}(M)$, we must have $\int_M g Lu dV=0$  for all  $g\in L^{\frac{p}{p-1}}(M)$. Hence $Lu=0$ and $u=0$ by the $\ker L=0$, which contradicts the condition $\|u\|_p=1$.

Combining (\ref{3}) and (\ref{4}) one can easily obtain (\ref{1}) and (\ref{2}) is the consequence of (\ref{1}) and Sobolev inequalities.
\end{proof}

We will also give an asymptotic theorem to describe the limiting behavior of $cu(x;c)$ as $c$ goes to negative infinity.  Here $u(x;c)$ is the solution of 
\begin{equation}\label{asymEqua}
\Delta_{d} u+\left\langle\alpha_{F}, d u\right\rangle-c u+f=0  
\end{equation}
 with $f \in  C^{\infty}(M)$ and $c<0$ on a compact manifold $M$. 

Assume $B_{1}$ and $B_{2}$ are Banach spaces and $B_{2}$ a subspace of $B_{1}$ such that the natural injection $B_{2} \rightarrow B_{1}$ is continuous, $\|\,\cdot\, \|_{1}$ is the norm in $B_{1}$, and $G: B_{2} \rightarrow B_{1}$ is a continuous linear map. First we need the asymptotic lemma in \cite{KW}.  

\begin{lemma}\label{asymptoticlemma}
(\cite{KW} Lemma 4.1)Assume that $(G+\alpha I): B_{2} \rightarrow B_{1}$ is invertible for all $\alpha \leq 0$ and that
\begin{equation}\label{asym}
\left\|(G+\alpha I)^{-1}\right\| \equiv \sup _{Z \in B_{1}} \frac{\left\|(G+\alpha I)^{-1} Z\right\|_{1}}{\|Z\|_{1}} \leq m(\alpha),
\end{equation}
where $m(\alpha) \rightarrow 0$ as $\alpha \rightarrow-\infty$. If $Y \in B_{2}$ (not just $B_{1}$), let $X_{\alpha}$ be the unique solution of $G X_{\alpha}+\alpha X_{\alpha}=Y$. Then
$$
\lim _{\alpha \rightarrow-\infty}\left\|\alpha X_{\alpha}-Y\right\|_{1}=0,
$$
that is, $\alpha X_{\alpha} \rightarrow Y$.
\end{lemma}

Then we give the asymptotic theorem which is a generalization of Theorem 4.4 in \cite{KW} by adding a gradient term.

\begin{theorem}\label{asymptotictheorem}
For $c<0$, $u(x ; c)$ is the unique solution of (\ref{asymEqua}) on a compact manifold $M$. Then
$$
\lim _{c \rightarrow-\infty} c u(x;c)=f(x),
$$
where the convergence is uniform on $M$.
\end{theorem}
\begin{proof}
We apply Lemma \ref{asymptoticlemma} with $B_{1}=W^{s, 2}(M)$ and $B_{2}=$ $W^{s+2,2}(M)$, for $s$ a positive even integer, $s=2 k$. Let $G=-(\Delta_{d}+\left\langle\alpha_{F}, d \, \cdot \,\right\rangle+I)$. Then $G+$ $\alpha I: W^{s+2,2} \rightarrow W^{s, 2}$ is continuous and has a continuous inverse for any $\alpha \leq 0$ by standard elliptic PDE theory. If we let $Z_2=(G+\alpha I)^{-1} Z_1$ for some $Z_1 \in W^{s, 2}$, then (\ref{asym}) will be established if we can show that there is a function $m(\alpha) \rightarrow 0$ as $\alpha \rightarrow-\infty$ such that
$$
\|Z_2\|_{s, 2} \leq m(\alpha)\|(G+\alpha I) Z_2\|_{s, 2}
$$
for all $Z_2 \in W^{s+2,2}$. Now the fundamental inequality for elliptic operators shows that the $W^{s, 2}$ norm of $\varphi$ is equivalent to the $L^{2}$ norm of $G^{k} \psi=$ $(-1)^{k}(\Delta_{d}+\langle\alpha_{F}, d  \, \cdot \,\rangle+I)^{k} \psi$ (recall that $\left.s=2 k\right)$. Thus we can consider the $W^{s, 2}$ inner product and norm to be defined by
$$
\langle Z_1, Z_2\rangle_{s, 2}=\langle G^{k} Z_1, G^{k} Z_2\rangle
$$
where $\langle \, \cdot \,, \, \cdot \,\rangle$ is the $L^{2}(M)$ inner product. Consequently, if we denote $G^{k} Z_2 = \hat{Z_2}$, we have
$$
-\langle G Z_2, Z_2\rangle_{s, 2}=-\langle G \hat{Z_2}, \hat{Z_2}\rangle
=-\langle -\Delta_{d}\hat{Z_2}-\langle\alpha_{F}, d \hat{Z_2}\rangle-\hat{Z_2} , \hat{Z_2}\rangle \ge 0
,$$
therefore for $\alpha<0$ we have
$$-\alpha\left\| Z_2\right\|_{s,2}^{2} \leq-\langle G Z_2+\alpha  Z_2,  Z_2\rangle_{s,2} \leq\left\|(G+\alpha I){ Z_2}\right\|_{s, 2}\| Z_2\|_{s, 2}.$$
This verifies (\ref{asym}) with $m(\alpha)=1 /|\alpha|$, that is,
$$
\| Z_2\|_{s, 2} \leq \frac{1}{|\alpha|}\|(G+\alpha I)  Z_2\|_{s,2} . 
$$
If we let $\alpha=c+1$, then $f=-\Delta_{d} u-\langle\alpha_{F}, du\rangle+c u=G u+\alpha u$. Thus by Lemma \ref{asymptoticlemma},
$$
\lim _{\alpha \rightarrow-\infty}\|\alpha u-f\|_{s, 2}=0
$$
for any positive integer s. By the Sobolev embedding theorem, if $s$ is sufficiently large, we conclude the convergence is uniform on $M$,
$$
\lim _{c \rightarrow-\infty}(c+1) u(x ; c)=f(x) .
$$
This completes the proof.
\end{proof}

\section{Kazdan-Warner type equation with $c<0$}\label{sec4}

We use the super and sub-solution method to deal with Kazdan-Warner type equation
$$
\Delta_{d} w+\left\langle\alpha_{F}, d w\right\rangle+c=\varphi e^{w}.
$$

Let's define the super and sub-solutions. We call $w_{-}, w_{+} $ a sub (respectively, super) solution of (\ref{KWte}) on $M$ if
$$
\Delta_{d} w_{-}+\left\langle\alpha_{F}, d w_{-}\right\rangle+c-\varphi e^{w_{-}}\leq 0,$$
$$
\Delta_{d} w_{+}+\left\langle\alpha_{F}, d w_{+}\right\rangle+c-\varphi e^{w_{+}}\geq 0.
$$

\begin{theorem}\label{thmUPPSUB}

If we successfully construct the super and sub-solutions $w_{+}$ and $w_{-}$, then (\ref{KWte}) has a smooth solution. 
\end{theorem}

\begin{proof}
The proof is a standard process to prove the existence of the solution when super and sub-solutions are constructed, which is similar to Theorem 2.3 in \cite{Fus}, here we rewrite it briefly. In the proof $C$ denotes a constant that may change from line to line, but is always independent of $w_{-}$, $w_{+}$ or $w_i$.($i=0,1,...$)

If we can construct the super and sub-solutions, we set
$$
s=\inf _{M} w_{-}, \quad S=\sup _{M} w_{+} ,
$$
and take $\lambda>0$ large so that
$$
\lambda>-\varphi e^{w} \quad \text { for any }(x, w) \in M \times[s, S].
$$

Now we write $w_{0}=w_{-} .$ For any $w_{i}, i=0,1, \ldots$, we suppose $w_{i+1} \in C^{\infty}(M)$ solve
\begin{equation}\label{equa9}
Kw_{i+1} := \Delta_{d} w_{i+1}+\left\langle\alpha_{F}, d w_{i+1}\right\rangle+\lambda w_{i+1}=\varphi e^{w_{i}}-c+\lambda w_{i} .
\end{equation}
Since $\lambda > 0$, the existence of smooth solution of this equation is guaranteed by standard elliptic PDE theory. We first prove
$$
w_{-} \leq w_{i} \leq w_{+},
$$
this is obviously true for $i=0$. Suppose it holds for some $i \geq 0$, consider $w_{i+1}$. First, we note
$$
\Delta_{d}\left(w_{i+1}-w_{-}\right)+\left\langle\alpha_{F}, d (w_{i+1}-w_{-})\right\rangle+\lambda\left(w_{i+1}-w_{-}\right) \geq \varphi e^{w_{i}}-\varphi e^{w_{-}}-\lambda\left(w_{-}-w_{i}\right),
$$
by the mean value theorem, $\varphi e^{w_{i}}-\varphi e^{w_{-}}-\lambda\left(w_{-}-w_{i}\right)\ge 0$, then by the maximum principle, $ w_{-} \le w_{i+1}$, similarly, we can prove $ w_{i+1} \leq w_{+}$, so $ w_{-} \le w_{i} \leq w_{+}$ for all $i$ by induction. In particular,
$$
s \leq w_{i}(x) \leq S \quad \text { for any }  i=0,1, \ldots
$$
Similarly, we can prove the monotonicity
\begin{equation}\label{mono}
w_{-} \leq w_{1} \leq w_{2} \leq \cdots \leq w_{+} .
\end{equation}

On the other hand, from Lemma \ref{lemma1} we have 
$$
\|w_{i}\|_{2, p}\le C, \quad \forall w_{i} \in W^{2, p}(M).
$$
By (\ref{2}), both $w_{i}$ and its first derivative are uniformly bounded, so Ascoli-Arzel\`a theorem implies that $$
w_i \rightarrow w \text { in } C^0(M).
$$ In view of the monotonicity (\ref{mono}), we conclude that the entire sequence $w_i$ itself converges uniformly to $w$. Inequality (\ref{3}) implies that
$$
\begin{aligned}
\left\|w_{i+1}-w_{j+1}\right\|_{2, p} & \leq C\left\|K\left(w_{i+1}-w_{j+1}\right)\right\|_p \\
& \leq C\left(\|\varphi\|_p\left\|e^{w_i}-e^{w_j}\right\|_{\infty}+\|\lambda\|_p\left\|w_i-w_j\right\|_{\infty}\right).
\end{aligned}
$$
Therefore $$
w_i \rightarrow w \text { in } W^{2, p}(M).
$$Since $K: W^{2, p} \rightarrow L^{p}$ is continuous, it follows that $w$ is a solution of (\ref{KWte}), using the Calder\`on-Zygmund inequality, Sobolev embeddings and bootstrap argument, we obtain that $w \in C^{\infty}(M)$.
\end{proof}

Then we give conditions for the existence of the solution of (\ref{KWte}), this is a generalization of Theorem 10.1 in \cite{KW}.

\begin{theorem}\label{thm1}
Let $(M,J,h)$ be a compact almost Hermitian manifold of real dimension $2n$, $\alpha_{F}$ is the Lee form. Consider $c<0$ and $ \varphi\in C^{\infty}(M)$. We denote $\frac{1}{\operatorname{Vol}(M)}\int_{M}\varphi dV$ as $\bar{\varphi}$.

(a) If a solution of (\ref{KWte}) exists, then $\bar{\varphi}<0$. And the unique solution $\phi$ of
\begin{equation}\label{necessaryEQUA}
\Delta_{d} \phi+\left\langle\alpha_{F}, d \phi\right\rangle-c \phi+\varphi=0
\end{equation}
must be positive.

(b) If $\bar{\varphi}<0$, then there is a constant $-\infty \leq c_{-}<0$ depending on $\varphi$ (which we write $\left.c_{-}(\varphi)\right)$ such that one can solve (\ref{KWte}) for all $c_{-}(\varphi)<c<0$ and can't solve for all $c<c_{-}(\varphi)$.
\end{theorem}
\begin{proof}
(a): Recall that the metric is of Gauduchon, the existence of smooth solution of (\ref{necessaryEQUA}) is guaranteed by standard elliptic PDE theory.

Using the substitution $v=e^{-w}$, one sees that (\ref{KWte}) has a solution $w$ if and only if there is a solution $v>0$ of
\begin{equation}\label{transformEQUA}
\Delta_{d} v+\left\langle\alpha_{F}, d v\right\rangle-c v+\varphi+\frac{|\nabla v|^{2}}{v}=0.
\end{equation}
Assume that $v$ is a positive solution of (\ref{transformEQUA}), let $\phi$ be the unique solution of (\ref{necessaryEQUA}) and let $w_{1}=\phi-v$. Then
$$
\Delta_{d} w_{1}+\left\langle\alpha_{F}, d w_{1}\right\rangle-c w_{1}=\frac{|\nabla v|^{2}}{v} \geq 0 .
$$
Simply using maximum principle we can obtain $\phi \geq v>0$. Consequently, a necessary condition for there to exist a positive solution of (\ref{transformEQUA}) is that the unique solution of (\ref{necessaryEQUA}) must be positive. This necessary condition immediately implies $\bar{\varphi}<0$ as one sees by integrating (\ref{necessaryEQUA}) since the metric is of Gauduchon.

(b): Let $v$ be the smooth solution of 
$$
\Delta_{d} v+\left\langle\alpha_{F}, d v\right\rangle=\varphi-\bar{\varphi},
$$
the solvability of this equation is guaranteed by Lemma \ref{sec2}.
Since $\left|e^{t}-1\right| \leq|t| e^{|t|}$, and since $\bar{\varphi}<0$, we can pick $a>0$ so small that
$$
\left|e^{a v}-1\right| \leq \frac{-\bar{\varphi}}{2\|\varphi\|_{\infty}} ,
$$
let $e^{b}=a$, if $c\ge \frac{a\bar{\varphi}}{2}$ and $w_{+}=a v+b$, we have
$$
\begin{aligned}
\Delta_{d} w_{+}+\left\langle\alpha_{F}, d w_{+}\right\rangle+c-\varphi e^{w_{+}}&\ge a \varphi \left(1-e^{a v}\right)-\frac{a \bar{\varphi}}{2} \geq 0 .
\end{aligned}
$$

Thus, with $c\ge \frac{a\bar{\varphi}}{2}$, we have a super solution $w_{+}$. Consequently, the critical $c_{-}(\varphi) \leq  \frac{a\bar{\varphi}}{2}<0$. For sub-solution, since $\varphi$ is bounded, we can let $w_{-}$  be a sufficiently small negative constant.

Clearly, if $w_{+}$ is a super solution for a given $c<0$, then $w_{+}$ is also a super solution for all $\tilde{c}<0$ such that $c \leq \widetilde{c}$. Therefore, there is a critical constant $-\infty \leq$ $c_{-}(\varphi) \leq 0$ such that we can construct super and sub-solutions to (\ref{KWte}) for negative $c$ with $c>c_{-}(\varphi)$ but can not construct super solution for $c<c_{-}(\varphi)$.
\end{proof}

\begin{remark}\label{re1}
Back to our problem, since we are dealing 
$$
\Delta_d w+\left\langle\alpha_F, d w\right\rangle+c(t)=\varphi e^w
$$
with $c(t)=\frac{2}{n t-t+1} \int_M s(t) d V$ and $\varphi=\frac{2 e^g}{n t-t+1} \hat{s}(t)$, here $g$ is the solution of 
$$
\Delta_d g+\left\langle\alpha_F, d g\right\rangle+\frac{2}{n t-t+1} s(t)=\frac{2}{n t-t+1} \int_M s(t) d V \text {. }
$$
So for our problem, the critical point $c_{-}(\varphi)$ in Theorem \ref{thm1} is actually $c_{-}(\frac{2 e^g}{n t-t+1}\hat{s}(t))$, we denote it as $c(\hat{s}(t))$, which means for each given smooth function $\hat{s}(t)$, there is an associated constant $c(\hat{s}(t))$. If $c(\hat{s}(t))<c(t)<0$, then there is a Hermitian metric conformal to $h$ with Gauduchon scalar curvature $\hat{s}(t)$. Whereas $c(t)<c(\hat{s}(t))$, it is impossible that $\hat{s}(t)$ is the Gauduchon scalar curvature function of a Hermitian metric in the conformal class of $h$.
\end{remark}

Then we start analysing the critical constant $c_{-}(\varphi)$, we have the following theorem by using the methods of Theorem 10.5 in \cite{KW}. In \cite{Fus}, E.Fusi proved part of (a) that when $\varphi(x) \leq 0$ for all $x \in M$ but $\varphi \not \equiv 0$, $c_{-}(\varphi)=-\infty$ by continuity path, we proved in (a) that the opposite direction is also correct, this theorem is a generalization of both Theorem 10.5 in \cite{KW} and Theorem 2.5 in \cite{Fus}.

\begin{theorem}\label{cphs}
We assume that $\varphi \in C^{\infty}(M)$.

(a) $c_{-}(\varphi)=-\infty$ if and only if $\varphi(x) \leq 0$ for all $x \in M$ but $\varphi \not \equiv 0$.

(b) If $\tilde{\varphi} \leq \varphi$, then $c_{-}(\tilde{\varphi}) \leq c_{-}(\varphi)$. Also, $c_{-}(\varphi)=c_{-}(\lambda \varphi)$ for any constant $\lambda>0$.

(c) Given $c<0$, there is a $\varphi$ with $\bar{\varphi}<0$ such that $c<c_{-}(\varphi)$. Thus, the necessary condition $\bar{\varphi}<0$ is not sufficient for solvability of (\ref{KWte}) and the critical constant $c_{-}(\varphi)<0$ can be made arbitrarily close to 0 .

(d) Let $\varphi \in L_{p}(M)$ for some $p>\operatorname{dim} M$. If there exists a constant $\alpha<0$ and $f \in L_{p}(M)$ such that $\varphi \leq f$ and $\|f-\alpha\|_{p}<-\alpha / \gamma(1-2 c)$, where $\gamma$ is the constant in (\ref{2}) with $$L v \equiv \Delta_{d} v+\left\langle\alpha_{F}, d v\right\rangle-c v$$ then there is a solution $w \in W^{2, p}(M)$ of (\ref{KWte}) with $w$ smooth when $\varphi$ is smooth.
\end{theorem}

\begin{proof}
(a): First we show that if $\varphi(x) \leq 0$ for all $x \in M$, but $\varphi \not \equiv 0$, then (\ref{KWte}) is solvable for all $c<0$. In view of the Theorem \ref{thmUPPSUB}, the solvability of (\ref{KWte}) is equivalent to the existence of a super solution $w_{+}$. Let $$\Delta_{d} v+\left\langle\alpha_{F}, d v\right\rangle=\varphi-\bar{\varphi}$$ and note that $\bar{\varphi}<0$. Pick constants $a$ and $b$ such that $a \bar{\varphi}<c$ and $\left(e^{a v+b}-a\right)>0$. Then let $w_{+}=$ $a v+b$. Since $\varphi \leq 0$, we find that
$$
\Delta_{d} w_{+}+\left\langle\alpha_{F}, d w_{+}\right\rangle+c-\varphi e^{w_{+}}>0 .
$$
Therefore $w_{+}$ is a super solution. Consequently $c_{-}(\varphi)=-\infty$ if $\varphi \leq 0$ but $\varphi \not \equiv 0$.

Then we show that $c_{-}(\varphi)=-\infty$ also implies $\varphi(x) \leq 0$ for all $x \in M$ but $\varphi \not \equiv 0$, we prove by contradiction. Suppose that $\varphi \left(x_{0}\right)>0$ for some $x_{0} \in M$, then using Theorem \ref{asymptotictheorem}, the unique solution of (\ref{necessaryEQUA}) is negative at $x_{0}$ for $c$ sufficiently large enough negative, which contradicts to Theorem \ref{thm1}.

(b): If $w_{+}$ is a super solution for $\varphi$, then it is also a super solution for any $\tilde{\varphi} \leq \varphi$. Therefore $c_{-}(\tilde{\varphi}) \leq c_{-}(\varphi)$. To see that $c_{-}(\varphi)=c_{-}(\lambda \varphi)$ for any constant $\lambda>0$, we only have to note that if $w$ is a solution of $$\Delta_{d} w+\left\langle\alpha_{F}, d w\right\rangle+c=\varphi e^{w},$$ then $v=w-\log \lambda$ is a solution of $$\Delta_{d} v+\left\langle\alpha_{F}, d v\right\rangle+c=\lambda \varphi e^{v}. $$ 

So for prescribed $c<0$ one can solve (\ref{KWte}) given the function $\lambda \varphi$ if and only if one can solve given the function $\varphi$.

(c): With $c<0$ given, we show there is a $\varphi \in C^{\infty}(M)$ with $\bar{\varphi}<0$ for which (\ref{KWte}) is not solvable, thus $c<c_{-}(\varphi)<0$. Let $\psi \in C^{\infty}(M)$ satisfy $\bar{\psi}=0$, but $\psi \not \equiv 0$. Choose a constant $\alpha>0$ so small that $\psi+\alpha$ still changes sign. Set $\varphi$ by  $$\varphi=-\Delta_{d} \psi-\left\langle\alpha_{F}, d \psi\right\rangle+c(\psi+\alpha),$$
then $\bar{\varphi}=c \alpha<0$. But by Theorem \ref{thm1} we conclude that (\ref{KWte}) has no solution for this $\varphi$ and this $c<0$ since the unique solution of 
$$
\Delta_{d} \phi+\left\langle\alpha_{F}, d \phi\right\rangle-c \phi+\varphi=0
$$
is $\phi=\psi+\alpha$, which changes sign.

(d): In view of the Theorem \ref{thmUPPSUB},in order to prove existence for 
\begin{equation}\label{super}
\Delta_{d} w+\left\langle\alpha_{F}, d w\right\rangle+c=\varphi e^{w},
\end{equation}
 it is sufficient to prove there is a super solution $w_{+}$ of (\ref{super}), that is
\begin{equation}\label{super w}
\Delta_{d} w_{+}+\left\langle\alpha_{F}, d w_{+}\right\rangle+c-\varphi e^{w_{+}} \geq 0 .
\end{equation}
Similar to (\ref{transformEQUA}), make the change of variable $v=e^{-w_{+}}$, then $w_{+}$ will satisfy (\ref{super w}) if $v$ is a positive solution of
\begin{equation}\label{super v}
\Delta_{d} v+\left\langle\alpha_{F}, d v\right\rangle-c v+\varphi+\frac{|\nabla v|^{2}}{v} \leq 0 .
\end{equation}
Let $\delta=-\alpha /(1-2 c)$, $L v \equiv \Delta_{d} v+\left\langle\alpha_{F}, d v\right\rangle-c v$ and $v$ be the unique solution of
$$
L v +f+\delta = 0 .
$$
Observe that $\psi \equiv 2 \delta$ is a solution of $L \psi=-2 c \delta$. Thus
$$
L(v-\psi)=2c\delta-f-\delta=\alpha-f,
$$
so inequality (\ref{2}) applied to $v-\psi$ reveals that
$$
 \quad\|v-2 \delta\|_{\infty}+\|\nabla v\|_{\infty} \leq \gamma\|L(v-\psi)\|_{p}=\gamma\|\alpha-f\|_{p}<\delta.
$$
In particular,
$$
\|v-2 \delta\|_{\infty}<\delta  \text { so } v(x)>\delta \text {. }
$$
Also, $\|\nabla v\|_{\infty}<\delta$ so
$$
\frac{|\nabla v|}{v}<1 \text { and } \frac{|\nabla v|^{2}}{v}<\delta.
$$

Therefore $v(x)>\delta>0$ and (\ref{super v}) is satisfied since
$$
\Delta_{d} v+\left\langle\alpha_{F}, d v\right\rangle-c v+\varphi+\frac{|\nabla v|^{2}}{v} \leq \varphi+\delta-f-\delta\leq 0.
$$
Thus ($\ref{super}$) has a solution, which means that for a given $c$, one can solve ($\ref{super}$) for a large class of $\varphi$.

\end{proof}

\begin{remark}\label{re2}
Back to our problem. That's to say, under the assumption $c(t)<0$, $c(\hat{s}(t))=-\infty$ if and only if the nonzero function $\frac{\hat{s}(t)}{n t-t+1}$ is nonpositive. Moreover, $\int_M \frac{2 e^g}{n t-t+1} \hat{s}(t) d V<0$ is a necessary condition for $\hat{s}(t)$ to be the Gauduchon scalar curvature of a hermitian metric conformal to $h$, but it's not sufficient.  As for (d),  we do not give some application here. However it is interesting to find examples of almost Hermitian manifold $M$ and show the existence of $f$ and $\alpha$.  For example, let  $M=K\times K$ and $K$ is a Riemannian surface with negative Euler character, and then one can find some diffeomorphism $\phi: M\to M$ to make $c(\phi^*\hat{s}(t))< c(t)$ as the surface case done by Kazdan and Warner in  in \cite{KW}.  

\end{remark}

\section{Other approaches}\label{sec5}
When $c=0$, the necessary condition of the existence of
\begin{equation}\label{KWc=0}
\Delta_{d} w+\left\langle\alpha_{F}, d w\right\rangle=\varphi e^{w}
\end{equation}
could be obtained by integrating on both side, that is $\varphi$ must change sign. When $\varphi =0 $, obviously it can be solved since it's linear. The sufficient condition for the existence of solution for (\ref{KWc=0}) on manifold with real dimension greater than $2$ remains an open problem.

When $c>0$, since a family of Lipschitz functions were constructed to prove that when using variational method, the lower bound does not exist \cite{CSZ}, there are several studies try to solve this case using flow approach such as  \cite{CSZ}, \cite{Ho}, \cite{LeM} and \cite{WY}. Here we use fixed point theorem and implicit function theorem to study it under some conditions on $s(t)$ and $\hat{s}(t)$ but regardless of the sign of $c(t)$. For this part, we directly study the original equation 
\begin{equation}\label{small}
\Delta_{d} u+\langle\alpha_{F}, d u\rangle+\frac{2}{n t-t+1}s(t)= \frac{2\hat{s}(t)}{n t-t+1}e^{u}.
\end{equation}

First we use the fixed point theorem to study the situation when $s(t)$ and $\hat{s}(t)$ are close enough in $C^{0,\alpha}$ norm and $\hat{s}(t)$ is in the set
$$
 \mathcal{I}=\left\{\varphi \in C^{\infty}(M) \mid Lu = \Delta_du +\langle\alpha_F, du\rangle- \frac{2}{n t-t+1} \varphi u, L \ is \ invertible.  \right\}.
$$
Notice that the $\mathcal{I}$ is a $G_{\delta}$-dense set by Fredholm alternative \cite[Theorem 5.15]{GT} which is related to the eigenvalues of $L$. We used the methods of Lemma 3.1 in \cite{LeM}.
\begin{theorem}\label{smallosc}
\textbf{Small Oscillation Case}
Let $(M,J,h)$ be an almost Hermitian manifold of real dimension $2 n$ endowed with a Hermitian metric.
Set  $B_{\varepsilon_{0}}( s(t))=\left\{u \in C^{0, \alpha}(M) \mid\| s(t)-u\|_{C^{0, \alpha}}<\varepsilon_{0}\right\} $, then there exists $\varepsilon_{0}>0$ and a $G_{\delta}$-dense set
$$ \mathcal{A} = \mathcal{I} \cap B_{\varepsilon_{0}}( s(t)) $$
such that if $ \hat{s}(t) \in \mathcal{A}$, then $\hat{s}(t)$ is the Gauduchon scalar curvatures of a Hermitian metric conformal to $h$. 
\end{theorem}
\begin{proof}
Rewrite (\ref{small}) as
$$
Lu= \frac{2}{nt-t+1} [\hat{s}(t)- s(t)-\hat{s}(t)\left(1+ u-e^{u}\right)].
$$

Choose $\varepsilon$ small enough such that if  $\|u\|_{C^{0, \alpha}}<\varepsilon$, $$| e^{ u}-1-u | \leqslant \varepsilon^2$$ and 
$$
\left\| \hat{s}(t)\left(e^{ u}-1-u\right)\right\|_{C^{0, \alpha}}<C\varepsilon ^{2}.
$$
 If $ \hat{s}(t) \in  \mathcal{I}$,  then  $L$ is invertible. We consider the operator
$$
T(u)=L^{-1}(\frac{2}{n t-t+1}[\hat{s}(t)-s(t)-\hat{s}(t)\left(1+u-e^u\right)]),
$$
Choose $\| \hat{s}(t)- s(t)\|_{C^{0, \alpha}}<\varepsilon_{0} $ small enough such that,  by Schauder's estimates, 
$$
\|T(u)\|_{C^{2, \alpha}} \leq \|L^{-1}\| \|\hat{s}(t)-s(t)-\hat{s}(t)\left(1+u-e^u\right)\|_{C^{0, \alpha}}< \varepsilon,
$$
which shows that
$$
T\left(B_{\varepsilon}\right) \subset B_{\varepsilon}
$$
where $$B_\varepsilon=\{u\in C^{2,\alpha}(M) \mid \|u\|_{ C^{2,\alpha}}<\varepsilon\}. $$
Moreover,
$$
\|T(u)-T(v)\|_{C^{2, \alpha}}=\left\|O\left(u^{2}\right)-O\left(v^{2}\right)\right\|_{C^{2, \alpha}} \leq C \varepsilon\|u-v\|_{C^{2, \alpha}}.
$$
So if $\varepsilon$ is small enough, $C\varepsilon<1$ and then $T$ is contracted, we do have a fixed point thus a solution to (\ref{small}).   Notice $\varepsilon_0$ depends on both the invertibility and the boundedness  of $L$. 

In sum, if $\hat{s}(t)\in \mathcal{A}$, then $\hat{s}(t)$ is the Gauduchon scalar curvatures of a Hermitian metric conformal to $h$.
\end{proof}

We can also use implicit function theory to study the situation when the $\mathcal{C}^{0, \alpha}$ norm of  $s(t)$, $\hat{s}(t)$ are both smaller than a constant. In this case we don't need to choose $\hat{s}(t) \in \mathcal{I}.$
\begin{theorem}\label{implicit}
\textbf{Implicit Theorem Approach}
Let $(M,J,h)$ be an almost Hermitian manifold of real dimension $2 n$ endowed with a Hermitian metric. Then there exists $\varepsilon>0$, depending just on $M$ and $h$, such that, if $\left\|s(t)\right\|_{\mathcal{C}^{0, \alpha}(M)}$ and $\left\|\hat{s}(t)\right\|_{\mathcal{C}^{0, \alpha}(M)}$ are both less than $\varepsilon$, for some $\alpha \in(0,1)$, then $\hat{s}(t)$ is the Gauduchon scalar curvatures of a Hermitian metric conformal to $h$. 
\end{theorem}
\begin{proof}
We use the method of Theorem 5.9 in \cite{ACS}.

Still  consider the equation \eqref{small}. Fix $\alpha \in(0,1)$.  Consider the Banach manifolds
$$
\mathcal{X}:=\left\{(u, \hat{s}(t),s(t)) \in C^{2,\alpha}(M)\times  C^{0,\alpha}(M)\times C^{0,\alpha}(M)\mid   \int_{M}  \hat{s}(t) e^{u}-s(t) dV=0\right\} ,
$$
 and 
$$
\mathcal{Y}:=\left\{u \in C^{0, \alpha}(M) \mid \int_{M} u d V=0\right\} .
$$
Then the map
$$
F: \mathcal{X} \rightarrow \mathcal{Y}, \quad F(u, \hat{s}(t), s(t))=\Delta_{d} u+\left\langle\alpha_{F}, d u\right\rangle + \frac{2}{n t-t+1}(s(t)-\hat{s}(t) e^{u}).
$$
satisfies $F(0,0,0)=0$ and 
$$
\left.\frac{\partial F}{\partial u}\right|_{(0,0,0)}=\Delta_{d} +\left\langle\alpha_{F}, d \, \cdot \,\right\rangle. 
$$
As show in Lemma \ref{sec2},   $\frac{\partial F}{\partial u}|_{(0,0,0)} $ is invertible on $ \mathcal{Y}$. Hence, by applying the implicit function theorem, there exists a neighbourhood $U$ of $(0,0)$ in $C^{0, \alpha}(M) \times C^{0, \alpha}(M)$ and a $C^{1}$ function $\tilde{u}$ such that
$$
F(\tilde{u}(\hat{s}(t),s(t)),\hat{s}(t),s(t))=0 .
$$
Take $\varepsilon>0$, depending on $M$ and $h$, such that
$$
\{\hat{s}(t) \mid \left\|\hat{s}(t)\right\|_{C^{0, \alpha}(M )}<\varepsilon\} \times\left\{ s(t)  \mid\| s(t)\|_{C^{0, \alpha}(M )}<\varepsilon\right\} \subset U .
$$
Then since $\left(\hat{s}(t), s(t)\right) \in U$, take $u:=\tilde{u}\left(\hat{s}(t), s(t)\right) \in C^{2, \alpha}(M )$ as above. By regularity, $u$ belongs in fact to $C^{\infty}(M )$. 
\end{proof}

\end{document}